\documentclass[reqno]{amsart}
\usepackage{enumerate}
\usepackage{amsmath,amssymb,amsthm,amsxtra,amsbsy}
\usepackage{a4}
\usepackage[mathscr]{eucal}
\usepackage[all]{xy}

\newcommand{\vanish}[1]{\relax}

\newcommand{\sector}[1]{S_{#1}}

\newcommand{\HT}{\mathscr{H}}
\newcommand{\PV}{\mathrm{PV}}

\newcommand{\Mlt}{\mathcal{M}}

\newcommand{\DD}{\mathscr{D}}

\newcommand{\RR}{\mathscr{R}}

\newcommand{\resolv}{\varrho}

\DeclareMathOperator{\Cos}{Cos}
\DeclareMathOperator{\Sin}{Sin}

\DeclareMathOperator{\calE}{\mathcal{E}}
\DeclareMathOperator{\calF}{\mathcal{F}}

\newcommand{\suchthat}{\,\,|\,\,}

\newcommand{\Sum}[2][\relax]{%
 \ifx#1\relax \sideset{}{_{#2}}\sum 
 \else \sideset{}{^{#1}_{#2}}\sum
 \fi}

\renewcommand{\phi}{\varphi}
\renewcommand{\epsilon}{\varepsilon}

\newcommand{\N}{\mathbb{N}}

\newcommand{\R}{\mathbb{R}}
\newcommand{\C}{\mathbb{C}}

\newcommand{\ohne}{\setminus}

\makeatletter
\newcommand{\set}[2][\relax]{%
  \ifx#1\relax \ensuremath{
  \left\lbrace#2\right\rbrace}
  \else \ensuremath{%
  \setbox0=\hbox{\ensuremath{#2}}
  \dimen@\ht0
  \advance\dimen@ by \dp0
  \left\lbrace\left.#1\,\rule[-\dp0]{0pt}{\dimen@}\right|#2\right\rbrace}
  \fi}
  \makeatother

\newcommand{\pfeil}{\longrightarrow}
\newcommand{\tpfeil}{\longmapsto}

\newcommand{\car}{\mathbf{1}}

\DeclareMathOperator{\re}{Re}
\DeclareMathOperator{\im}{Im}

\DeclareMathOperator{\sgn}{sgn}

\newcommand{\cls}[1]{\overline{#1}}
\newcommand{\abs}[1]{\left| #1 \right|}

\newcommand{\rand}{\partial}

\renewcommand{\abs}[1]{\left\vert#1\right\vert}

\DeclareMathOperator{\Lin}{\mathcal{L}}

\newcommand{\norm}[2][\relax]{%
   \ifx#1\relax \ensuremath{\left\Vert#2\right\Vert}
   \else \ensuremath{\left\Vert#2\right\Vert_{#1}}
   \fi}

\makeatletter
\newcommand{\sprod}[2]{\ensuremath{%
  \setbox0=\hbox{\ensuremath{#2}}
  \dimen@\ht0
  \advance\dimen@ by \dp0
  \left(\left.#1\rule[-\dp0]{0pt}{\dimen@}\,\right|#2\hspace{1pt}\right)}}
 \makeatother

\newcommand{\Borel}{\mathfrak{B}}

\newcommand{\spacefont}{\mathbf}

\newcommand{\Ct}[2][\relax]{%
 \ifx#1\relax \ensuremath{ {\mathbf{C}}^{\mathbf{#2}}  }
 \else {\mathbf{C}}^{\mathbf{#2}}_{\boldsymbol{#1}}
 \fi}

\newcommand{\Bt}[2][\relax]{%
 \ifx#1\relax \ensuremath{ {\mathbf{B}}^{\mathbf{#2}}  }
 \else {\mathbf{B}}^{\mathbf{#2}}_{\boldsymbol{#1}}
 \fi}

\newcommand{\ct}[1][\relax]{%
 \ifx#1\relax \mathbf{c}
 \else \mathbf{c}_{\boldsymbol{#1}}
 \fi}

\newcommand{\Ell}[2][\relax]{%
   \ifx#1\relax \mathbf{L}^{\boldsymbol{#2}}
   \else \mathbf{L}^{\boldsymbol{#2}}_{\boldsymbol{#1}}
   \fi}
\newcommand{\Wee}[2][\relax]{%
   \ifx#1\relax \mathbf{W}^{\boldsymbol{#2}}
   \else \mathbf{W}^{\boldsymbol{#2}}_{\boldsymbol{#1}}
   \fi}

\newcommand{\Har}[2][\relax]{%
   \ifx#1\relax \boldsymbol{\mathsf{H}}^{\boldsymbol{#2}}
   \else \boldsymbol{\mathsf{H}}^{\boldsymbol{#2}}_{\boldsymbol{#1}}
   \fi}

\newcommand{\eM}{\spacefont{M}}

\newcommand{\Fourier}{\mathcal{F}}
\newcommand{\fourier}[1]{\widehat{#1}}

\newcommand{\sfunk}{\boldsymbol{\mathcal{S}}} 
\newcommand{\distr}[1][\relax]{%
    \ifx#1\relax \spacefont{D}
    \else \spacefont{D}^{\mathbf{#1}}
    \fi}

\newcounter{aufzi}
\newenvironment{aufzi}{\begin{list}{ {\upshape\alph{aufzi})}}{
        \usecounter{aufzi}
        \topsep1ex
        \parsep0cm
        \itemsep1ex
        \leftmargin1cm
        \labelwidth0.5cm
        \labelsep0.3cm
}}
{\end{list}}

\newcounter{aufzii}

\newcounter{aufziii}
\newenvironment{aufziii}{\begin{list}{ {\upshape\arabic{aufziii})}}{
        \usecounter{aufziii}
        \topsep1ex
        \parsep0cm
        \itemsep1ex
        \leftmargin0.8cm
        \labelwidth0.5cm
        \labelsep0.3cm
}}
{\end{list}}

\newcounter{aufziv}

\newcommand{\ccalT}{\mathcal{T}}
\newcommand{\diff}[1]{\mathrm{d}#1}
\newcommand{\name}{\textsc}

\newcommand{\expo}[1]{\mathrm{e}^{#1}}

\newcommand{\CT}{\mathcal{C}}

\newtheorem{thm}{Theorem}[section]
\newtheorem{lemma}[thm]{Lemma}
\newtheorem{prop}[thm]{Proposition}

\theoremstyle{definition}

\newtheorem{rem}[thm]{Remark}

\newlength{\usualparindent}
\begin{document}

\title[The Group Reduction for Bounded Cosine Functions]{The Group Reduction for
Bounded Cosine Functions on UMD Spaces}
\author{\textsc{Markus Haase}}

\address{Delft Institute of Applied Mathematics, Technical University Delft,
PO Box 5031, 2600 GA Delft, The Netherlands} 
\email{m.h.a.haase@tudelft.nl} 

\thanks{Author's e-mail address: {\tt m.h.a.haase@tudelft.nl}}

\thanks{This work was completed with the support of the EU Marie Curie
``Transfer of Knowledge'' project ``Operator Theory Methods for Differential
Equations''.}   
\date{5 September 2007}

\begin{abstract}
{\sffamily It is shown that if $A$ generates a bounded cosine operator
function on a UMD space $X$, then $i(-A)^{1/2}$ generates a bounded $C_0$-group.
The proof uses a transference principle for cosine functions.}


\end{abstract}

\renewcommand{\subjclassname}{\textup{2000} Mathematics Subject Classification}
\maketitle

\section{Introduction}\label{s.intro}

A cosine function on a (complex) Banach
space $X$  is a strongly continuous mapping
$\Cos : \R \pfeil \Lin(X)$ that satisfies the identity
\begin{equation}\label{int.e.cos}
 \Cos(t +s) + \Cos(t-s) = 2\Cos(t) \Cos(s) \quad \quad (t,s\in \R)
\end{equation}
as well as $\Cos(0) = I$.
One can prove from this that a cosine function is exponentially bounded, e.g.
\[ \theta(\Cos) := \inf \big\{ \omega \ge 0 \suchthat \exists M \ge 1 :
\norm{\Cos(t)} \le M\expo{\omega \abs{t}}, t \in \R\big\} < \infty. 
\]
The {\em generator} of a cosine function $\Cos$ is defined as the unique
operator $A$ such that 
\begin{equation}\label{int.e.gen}
 \lambda R(\lambda^2,A) = \int_0^\infty \expo{-\lambda t} \Cos(t)\,\diff{t}
\quad \quad (\lambda > \theta(\Cos)).
\end{equation}
The generator $A$ is densely defined, and 
the cosine function provides solutions to the
second-order abstract Cauchy problem
\[   u''(t) = Au, \quad \quad u(0)= x, \quad u'(0)=0.
\]
Conversely, if the abstract second order problem for an operator is
well-posed, then it gives rise to a cosine function. In this way,
cosine functions play the same role for the second order problem as semigroups
do for the first order problem. We refer to \cite[Chapter 3.14-3.16]{ABHN}
for more on the theory of cosine functions.

\medskip

An important example of a cosine function arises as
\[ \Cos(t) = \frac{1}{2} ( U(t) + U(-t))\quad \quad (t\in \R),
\]
where $U$ is a $C_0$-group. The generator $A$ of the cosine function
and the generator $B$ of the group are then related by $A = B^2$.
(As an example consider $B=\diff{}/\diff{t}$ 
the generator of the shift group on 
$\Ell{2}(\R)$; then $A$ is the one-dimensional Laplacian.)
It is natural to ask which cosine functions arise in this manner, but
in general there is little hope. 
Indeed, there is no way in general to reconstruct a group $U$ from 
its associated cosine function.
Taking squares deletes information that cannot be recovered. However,
as $B$ is a square root of $A$, one might look at $i(-A)^{1/2}$,
whenever $-A$ is sectorial. (The minus sign is natural here, since
the spectrum of $A$ extends to the left.) Let us call the group
generated by $i(-A)^{1/2}$, if it exists, the {\em square root reduction
group} associated with the original cosine function.
It turns out that  in general the square root reduction group does not exist, 
but only due to a shortcoming of the Banach space.
Indeed, \name{Fattorini} \cite{Fat69b} has shown the following result.

\medskip

\noindent{\bf Theorem.}\ 
{\slshape Let $A$ be the generator of a cosine function on a UMD space $X$.
If $-A$ is also sectorial, then $i(-A)^{1/2}$ generates a $C_0$-group.}

\medskip

A proof adapted from the orginal one is in \cite[3.16.7]{ABHN}.
More recently, functional calculus methods have been used to give a different
proof (see \cite{Haa06dpre} and combine it with \cite[Proposition 3.17]{HaaFC}
and standard perturbation). However, the approach was via the so-called
phase space and a direct functional calculus proof is still to be found.

\medskip

There is another issue here. Fattorini's theorem is qualitative in nature
and tells us nothing about the growth properties of the reduction group, depending
on the growth of the cosine function. In particular, it has been an open
problem for a long time 
whether the reduction group associated with a {\em bounded} cosine function is
itself bounded. On Hilbert spaces this is known to be true \cite{Fat70/71},
but the methods used are typical for Hilbert spaces, finding
self-adjointness by introducing an equivalent scalar product. 
To the bst of our knowlege, the last serious attempt to solve 
the problem on a general UMD space
was made by \name{Cioranescu} and \name{Keyantuo} in \cite{CioKey01}.
The present paper solves the problem in the affermative. 

\begin{thm}\label{int.t.main}
Let $A$ generate a bounded cosine function on a UMD space. Then
$i(-A)^{1/2}$ generates a bounded $C_0$-group.
\end{thm}

To avoid many minus signs, it is convenient to change notation a little
and write $A$ instead of $-A$. Moreover, we shall prove Theorem \ref{int.t.main}
in two steps according to the  following equivalent version:

\setcounter{thm}{0}
\begin{thm}[{\bf Alternative Version}]\label{int.t.alt}
Let $-A$ be the generator of a bounded cosine function $\Cos$ on a UMD space $X$. 
Then there exists a bounded $C_0$-group $U$ such that 
\[ \Cos(s)  = \frac{1}{2} (U(s) + U(-s)) \quad \quad (s\in \R)
\]
and $-iA^{1/2}$ is the generator of $U$.
\end{thm}

The purpose of this reformulation 
is to avoid the theory of sectorial operators
as long as possible. In fact, our  proof of the first part (the 
existence and boundedness of the group) is 
essentially self-contained, whereas for the second part
(the identification of the generator) 
one has to appeal to the well-established theory of fractional powers. 
However, we find it very instructive that this theory is not needed to 
obtain the existence and boundedness of the  
group reduction in the first place. 
Needless to say that we will not make use of Fattorini's original theorem at
any place.

The paper is structured as follows. In Section \ref{s.fc1} we construct a functional
calculus $\Phi$ in a Phillips type manner using integrals over the cosine function.
We show how this functional calculus can be interpreted in a canonical way
as a functional calculus for a certain unbounded operator $B$. The group $U$
will be given as $U(s)= \Phi(e^{-is\cdot}) = e^{-isB}$, $s\in \R$, and the
(uniform) boundedness of these operators is reduced to the uniform boundedness
of certain approximants. By means of a transference principle, which we
state and prove in Section \ref{s.tp}, one reduces this uniform boundedness
to the uniform boundedness of certain Fourier-multipliers  on the space 
$\Ell{2}(\R;X)$. In Section \ref{s.umd} we shortly provide the necessary
background on Fourier multipliers and the notion of UMD spaces, and finally
bring the different ingredients together to prove that the group $U$
is uniformly bounded. In Section \ref{s.rep} we present still another approach
to this result, using a different approximation (Theorem \ref{rep.t.rep}).
Finally, in Section \ref{s.fc2} we prove that the operator $B$ constructed
before is identical to the square root $A^{1/2}$ obtained from $A$ by means
of the sectorial functional calculus (or by the classical theory of fractional
powers). Moreover, we show that the functional calculus $\Phi$ is compatible
with and provides a proper extension of the sectorial functional calculus for 
$A^{1/2}$.

\vanish{
We shall provide two different 
routes for proving Theorem \ref{int.t.main}, 
both relying crucially on a transference theorem for cosine functions
(Theorem \ref{tp.t.tpbc}). 
The first (quicker) route follows the lines of \cite{CioKey01} 
where the reduction group is found as the boundary value of
the holomorphic semigroup generated by $-(-A)^{1/2}$. 
Besides the mentioned transference result, it uses 
the functional calculus for sectorial operators and the theory
of boundary values for holomorphic semigroups from \cite[Chapter~3.9]{ABHN}.
However, since the functions $f_t(z) = \expo{itz^{1/2}}, 0\not=t \in \R$, are not 
bounded on sectors, the sectorial functional calculus is actually 
not the right
framework to prove even the boundedness of the operators
$f_t(A)$ (let alone the uniform boundedness).  
So in the second part of the paper 
we provide a suitable holomorphic functional
calculus leading to a second proof of Theorem \ref{int.t.main}
and to a more perspicious approach to a known representation
formula (see Theorem \ref{rep.t.rep}). 
}

\bigskip

\noindent
{\em Definitions and Conventions}\\
We usually consider (unbounded) closed
operators $A,B$ on a complex Banach space $X$.  By $\Lin(X)$ we denote
 the set of all {\em bounded} (fully-defined) operators on $X$.
The domain and the range
of a general  operator $A$ are denoted by $\DD(A)$ and $\RR(A)$, respectively. 
Its resolvent is $R(\lambda,A) = (\lambda-A)^{-1}$, and $\resolv(A)$
denotes the set of $\lambda \in \C$ where $R(\lambda,A) \in \Lin(X)$.
Its complement $\sigma(A) = \C\ohne \resolv(A)$ is the {\em spectrum}.

For each open subset $\Omega \subset \C$
we denote by $H^\infty(\Omega)$
the Banach algebra of bounded holomorphic functions on $\Omega$. 
If $\Omega$ is an arbitrary locally compact space, then 
the set of complex regular Borel measures on $\Omega$ is 
denoted by $\eM(\Omega)$.
The {\em Fourier transform} of a tempered distribution $\Phi$ on $\R$ is 
denoted by $\Fourier(\Phi)$ or $\fourier{\Phi}$. We often write
$s$ and $t$ (in the Fourier image) to denote the real coordinate, e.g.
$\sin t/ t$ denotes the {\em function} $t \tpfeil  \sin t/ t$. 
On a complex domain we use $z$ as the coordinate, so that 
$f(z)$ denotes the function $z \tpfeil f(z)$.

Let $X$ be a Banach space and $p \in [1, \infty)$. 
For a finite measure $\nu \in \eM(\R)$ we
denote by
\[ L_\nu := (f \tpfeil f \ast \nu): \Ell{p}(\R;X) \to \Ell{p}(\R;X)
\]
the convolution operator on the $X$-valued $\Ell{p}$-space.

\section{A Functional Calculus}\label{s.fc1}

We suppose in this section that $-A$ is the generator of a bounded
cosine function $(\Cos(t))_{t\in \R}$ on the Banach space $X$. 
Note that as mentioned in the introduction
we consider $-A$ instead of $A$ in order to avoid many minus signs later.
In order to construct the reduction group, we shall built up 
a functional calculus for an yet unknown operator $B$ and only later 
(in Section \ref{s.fc2}) we will identify
this operator as $A^{1/2}$.

\medskip

The so-called {\em Phillips calculus} for $\Cos$ is the mapping
\[ (\mu \tpfeil T_\mu): \eM(\R) \pfeil \Lin(X)
\]
with $T_\mu$ being defined by
\[ T_\mu x := \int_{\R} \Cos(s)x\, \mu(\diff{s}) \quad\quad(x\in X, \mu \in \eM(\R)).
\] 
Since $\Cos$ is an even
function, the operator $T_\mu$ depends only on the {\em even part} 
$\mu_e$ of $\mu$, defined by 
\[ \mu_e(A) := \frac{1}{2}(\mu(A) + \mu(-A)) \quad \quad (A\in \Borel(\R)),
\]
and hence we may assume always that $\mu$ is an even measure. Thus we let
\[ 
\eM_e(\R) := \{ \mu \in \eM(\R) \suchthat \text{$\mu$ is even}\}
\]
and note that this is a closed subalgebra of the convolution algebra
$\eM(\R)$ of all bounded measures on $\R$.

\begin{prop}\label{fc1.p.cvl}
Let $(\Cos(s))_{s\in \R}$ be a bounded cosine function. Then the mapping
\[
(\mu \tpfeil T_\mu): \eM_e(\R) \pfeil \Lin(X)
\]
is a homomorphism of algebras.
\end{prop}

\begin{proof}
Let $\mu, \nu \in \eM_e(\R)$ and $x\in X$. We compute
\begin{align*}
T_\mu T_\nu x & = \int_\R \int_\R   \Cos(t)\Cos(s)x\, \nu(\diff{s})\,\mu(\diff{t})
\\ & =
 \frac{1}{2}\int_\R \int_\R \Cos(t+s)x\, \nu(\diff{s})\,\mu(\diff{t}) 
+
\frac{1}{2} \int_\R \int_\R \Cos(t-s)x\, \nu(\diff{s})\,\mu(\diff{t})
\\ & \stackrel{(*)}{=}   \int_\R \int_\R \Cos(t+s)x\, \nu(\diff{s})\,\mu(\diff{t})
\\ &=  \int_\R \Cos(s)x \, (\mu \ast \nu)(\diff{s}) = T_{\mu \ast \nu} x.
\end{align*}
In equality $(*)$ we performed
a change of variable $s \tpfeil -s$ in the second integral and
used that $\nu$ is an even measure.
\end{proof}

For $\mu \in \eM(\R)$ we define its {\em cosine transform} by 
$\CT\mu := \Fourier(\mu_e)$, i.e.
\[ (\CT\mu)(t) =  \int_\R \cos(st)\, \mu(\diff{s}) 
\quad \quad (t\in \R)
\]
Evidently, if $\mu$ is even then $\CT\mu$ coincides 
with the Fourier transform of $\mu$. Moreover, $\CT\mu$ is always an even function,
hence is determined by its restriction to $\R_+$. Therefore often
we shall not distinguish between a function defined on $\R_+$ and its
even extension to $\R$. Let
\[ \calE(\R_+) := \{ \CT\mu \suchthat \mu \in \eM(\R)\} 
= \{ \Fourier(\mu) \suchthat \mu \in \eM_e(\R)\}
\]
Then $\calE(\R_+)$ is an algebra with respect to pointwise multiplication of 
functions. If $f = \CT\mu \in \calE(\R_+)$ we define $\Phi(f) := T_\mu$, 
which is a good definition since the Fourier transform is injective and 
$T_\mu= T_{\mu_e}$. The mapping
\[ \Phi : \calE(\R_+) \pfeil \Lin(X)
\]
is a homomorphism of algebras,
by Proposition \ref{fc1.p.cvl} and the product law of the Fourier transform.
Note that for $\lambda > 0$ and 
$\mu := (1/2)e^{-\lambda \abs{s}}\diff{s}$ we have 
\[ f(t) := \CT(\mu)(t) = \fourier{\mu}(t) = \frac{\lambda}{\lambda^2 +t^2}
\quad \quad (t\in \R)
\]
and
\[ \Phi(f) = T_\mu = \int_0^\infty e^{-\lambda s}\Cos(s)\, \diff{s} =
\lambda R(\lambda^2, -A) = \lambda(\lambda^2 + A)^{-1}.
\]
This is clearly an injective operator. Hence $(\calE, \Phi)$ is a proper
primary functional calculus, and we may choose freely any superalgebra
$\calF$ of $\calE$ to obtain a proper abstract functional calculus
in the sense of \cite[Chapter 1]{HaaFC}. We might take for example
the algebra of all functions from $\R_+$ to $\C$. A function 
$f : \R_+ \pfeil \C$ is called {\em regularizable} if there is
$e\in \calE(\R_+)$ such that $ef \in \calE(\R_+)$ as well and $\Phi(e)$ is
injective. In this case, $e$ is called a {\em regularizer} for $f$, and 
the (closed but potentially unbounded) operator $\Phi(f)$ is defined
as 
\[ \Phi(f) := \Phi(e)^{-1} \Phi(ef).
\]
This definition does not depend on the chosen regularizer and is also
compatible with the definition of $\Phi$ on the algebra $\calE(\R_+)$;
moreover, it obeys the standard rules for unbounded functional calculi.
See \cite[Chapter 1]{HaaFC} for proofs of these facts and more information. 

\medskip

To identify regularizable functions, we recall the well-known Bernstein lemma.

\begin{lemma}\label{fc1.l.bern}
Let $f \in H^1(\R)$, i.e.~$f, f' \in \Ell{2}(\R)$. Then $f \in \Fourier(\Ell{1}(\R))$.
Moreover, the mapping $\Fourier^{-1}: H^1(\R) \pfeil \Ell{1}(\R)$ is continuous.
\end{lemma}

\begin{proof}
See  \cite[Lemma~8.2.1]{ABHN}.
\end{proof}

Here is a first application.

\begin{lemma}
Let $f \in \Ct{1}(\R_+)$ such that $f'$ is polynomially bounded. Then 
$f$ is regularizable, whence $\Phi(f)$ is defined.
\end{lemma}

\begin{proof}
As usual we view $f$ as an even function on $\R$.
The hypotheses imply that $f$ is polynomially bounded. 
Let $g(t) := f(t) (1+t^2)^{-n}$ for $n\in \N$ large
enough such that $g \in \Ell{2}(\R)$ and $f'(t)(1+t^2)^{-n} \in 
\Ell{2}(\R)$. Then 
\[  g'(t) = \frac{(1+t^2)^{n}f'(t) - n f(t)(1+t^2)^{n-1}}{(1+t)^{2n}} = 
\frac{f'(t)}{(1+t^2)^n} - 
\frac{n}{(1+t^2)} g(t)  
\]  
for $t \not= 0$. Hence $g' \in \Ell{2}(\R)$ and by Bernstein's lemma 
it follows that $g\in \Fourier(\Ell{1}(\R))$. As $g$ is even,
$g \in \calE(\R_+)$ and since $\Phi\big( (1 + t^2)^{-n}\big) = (1 + A)^{-n}$
is injective, $f$ is regularizable.
\end{proof}

Using this lemma one sees that the function $f(t) = \abs{t}$ is regularizable by
the function $(1  + t^2)^{-1}$, and  hence the operator
\[ B := \Phi(f) = \Phi(\abs{t})
\]
is defined. From the definition it is immediate that $\DD(A) \subset \DD(B)$,
and so $B$ is densely defined. Moreover, for $\lambda \notin \R_+$, the function
$(\lambda - \abs{t})^{-1}$ is in $H^1(\R)$, whence in $\calE(\R_+)$. Therefore
\[ (\lambda - B)^{-1} = \Phi\left( \frac{1}{\lambda - \abs{t}}\right) \in \Lin(X).
\] 
This shows that $\sigma(B) \subset \R_+$.

\medskip

It is reasonable to say that the functional calculus $\Phi$
is a functional calculus {\em for} $B$ and write $f(B)$ instead of 
$\Phi(f)$. In Section \ref{s.fc2} we shall show that
$B = A^{1/2}$, but that is unimportant at the moment.  We note
that also the functions $f_s, s\in\R$, defined by
\[ f_s(t) := e^{-is\abs{t}} \quad \quad (t\in \R)
\]
satisfy the conditions of the lemma. Hence we obtain the operators
\[ U(s) := \Phi(e^{-is\abs{t}}) = [e^{-is\abs{t}}](B) \quad \quad (s\in \R).
\]
Our main goal is to show
that if $X$ is a UMD space, then  $U$ 
is a bounded $C_0$-group on $X$. Functional calculus
then yields that indeed 
\[ \Cos(s) = [\cos(t)](B) = \frac{1}{2}(e^{-is\abs{t}} + e^{is\abs{t}})(B)
= \frac{1}{2}(U(s) + U(-s)) \quad \quad (s\in \R)
\]
and the first step in the proof of Theorem \ref{int.t.main} is complete.

\begin{lemma}
If $U(s)= [e^{-is\abs{t}}](B)$ is a bounded operator for every $s\in \R$, then
$(U(s))_{s\in \R}$ is a $C_0$-group and its generator 
is $-iB$.
\end{lemma}

\begin{proof}
Suppose that the hypothesis of the lemma holds true. General functional calculus
theory yields that $U$ is a group. 
To prove that $U$ is strongly continuous
by classical semigroup
theory \cite[Theorem 10.2.3]{HilPhi} it suffices to show that the trajectories 
$U(\cdot)x$, $x\in X$,  are all measurable. Since $\DD(A)$ is dense in $X$, it
suffices to show that $U(\cdot)x$ is continuous for each $x\in \DD(A)$. 
Hence we consider the functions
\[ g_s(t) := \frac{e^{-is\abs{t}}}{1+t^2} \quad \quad (t, s\in \R)
\]
and by Bernstein's lemma  it suffices to show
that $(s\tpfeil g_s): \R \pfeil H^1(\R)$ is continuous. This is easy to see.

Now, let $-C$ be the generator of $U$.
To prove that $iB = C$, note that $\norm{g_s}_{H^1} = O(\abs{s})$.
Hence one can take Laplace transforms within $H^1(\R)$ and obtains
for large $\lambda > 0$
\begin{align*}
 (1 + A)^{-1}(\lambda + C)^{-1}
& =
\int_0^\infty e^{-\lambda s}U(s)(1+A)^{-1}\, \diff{s}
=
\int_0^\infty e^{-\lambda s} \Phi(g_s)\,\diff{s}
\\ & = 
\Phi\left( \int_0^\infty e^{-\lambda s} g_s\, \diff{s}\right)
= 
\Phi\left(\frac{1}{(1+t^2)(\lambda + i\abs{t})}\right)
\\ & = (1+A)^{-1} (\lambda + iB)^{-1}.
\end{align*}
This shows that $C = iB$.
\end{proof}

Finally, we state a ``convergence lemma'' for our functional calculus.

\begin{lemma}\label{fc1.l.cl}
Let $(f_\alpha)_\alpha$ be a net of continuous functions on $\R_+$ converging
pointwise to a function $f$ and
satisfying the following conditions:
\begin{aufziii}
\item $f_\alpha/(1+t^2) \in H^1(\R)$ for all $\alpha$.
\item $f_\alpha/(1+t^2) \to  f/(1+t^2) $ within $H^1(\R)$.
\end{aufziii}
Then $f_\alpha(B)x \to f(B)x$ for all $x\in \DD(A)$. If in addition
\begin{aufziii}
\setcounter{aufzi}{2}
\item $\sup_\alpha \norm{f_\alpha(B)} < \infty$,  
\end{aufziii}
then $f(B) \in \Lin(X)$ and $f_\alpha(B) \to f(B)$ strongly. 
\end{lemma}

\begin{proof}
Since $\DD(A)$ is dense in $X$,
it suffices to show that 
\[
f_\alpha(B)(1+A)^{-1} \to f(B)(1+A)^{-1}
\]
in norm. This is guaranteed by conditions 1) and 2) and Bernstein's lemma. 
\end{proof}

Note that the hypotheses of Lemma \ref{fc1.l.cl} imply that 
$f_\alpha \to f$ uniformly on compact subsets of $\R_+$. On the other hand,
the hypotheses 1) and 2) of Lemma \ref{fc1.l.cl} are clearly satisfied if one has
the following situation:
\begin{aufziii}
\item $f\in \Ct{1}(\R_+)$ and $f_\alpha \in \Ct{1}(\R_+)$ for all $\alpha$;
\item $f_\alpha \to f$ and $f_\alpha' \to f'$ uniformly on compact subsets of $\R_+$;
\item $\sup_\alpha \norm{f_\alpha}_\infty + \norm{f_\alpha'}_\infty < \infty$.
\end{aufziii} 
A special case of this situation is given, for fixed $s\in \R$,  by the functions
\[   f_\alpha(t) := e^{-(\alpha + is)\abs{t}} ,\quad f(t) := e^{-is\abs{t}}\quad \quad 
(t\in \R, 0 < \alpha \le 1)
\] 
viewed as a net for  $\alpha \searrow 0$. Let us define
\[ T_B(\lambda) := \Phi(e^{-\lambda \abs{t}})= [e^{-\lambda \abs{t}}](B) \quad \quad 
(\re\lambda > 0).
\]
The following proposition shows that things behave as expected.

\begin{prop}\label{fc1.p.hsg}
Let $U$ and $T_B$ be defined as above. Then the following assertions hold.
\begin{aufzi}
\item $\displaystyle 
T_B(\lambda)x = \frac{\lambda}{\pi} \int_0^\infty \frac{\Cos(s)x}{\lambda^2 + s^2}
\, \diff{s}\quad \quad (x\in X, \re\lambda > 0)$.
\item $T_B$ is a bounded holomorphic semigroup of angle $\pi/2$. 
\item $-B$ is the generator of $T_B$.
\item For each $s\in \R$, $U(s)$ is a bounded operator if and only if
\[ \sup_{0< \alpha \le 1} \norm{T_B(\alpha + is)} < \infty,
\]
and in this case
$\quad U(s)x \,=\, \lim_{\alpha \searrow 0}  T_B(\alpha + is)x \quad \quad (x\in X)$.
\end{aufzi}
\end{prop}

\begin{proof}
a) follows from the fact that
\[ \Fourier^{-1}(e^{-\lambda\abs{t}})(s) = \frac{\lambda}{\pi(\lambda^2 + s^2)}
=: g_\lambda(s) \quad \quad (s\in \R).
\]
b) follows since the mapping
\[ (\lambda \tpfeil g_\lambda): \{\re\lambda > 0\} 
\pfeil \Ell{1}(\R) 
\]
is holomorphic and 
\[ \norm{g_\lambda}_{\Ell{1}} = \frac{2}{\pi} \int_0^\infty 
\frac{1}{\abs{(\lambda/\abs{\lambda})^2 + s^2 }}\,\diff{s} 
\quad \quad (\re\lambda > 0).
\]
c)\ Note that $\norm{g_r}_{\Ell{1}} = 1$ for all $r > 0$. Hence one may
take Laplace transforms in $\Ell{1}$ and this shows that
\[ \int_0^\infty e^{-zr}T_B(r)\, dr  = \Phi\left(
\int_0^\infty e^{-zr}e^{-rt}\, \diff{r}\right)
= \Phi\left(\frac{1}{z+t}\right) = (z + B)^{-1}
\]
for $z > 0$. 

d).\ By abstract functional calculus we have
\[ T_B(\alpha + is) = [e^{-\alpha \abs{t}} e^{-is\abs{t}}](B) = U(s)T_B(\alpha)
\]
Since $T_B$ is a bounded semigroup, if $U(s)$ is bounded then the operators
$T_B(\alpha + is)$, $\alpha > 0$ are uniformly bounded. Conversely, supposing that 
these operators are uniformly bounded one can apply the convergence lemma
(Lemma \ref{fc1.l.cl})
and the remarks following it. 
\end{proof}

As a consequence we note that to prove the first part of Theorem \ref{int.t.main}
we only have to establish the uniform boundedness
\[ \sup \{ \norm{T_B(\lambda)} \suchthat \re\lambda > 0\} < \infty.
\]
This will be done with the help of a transference principle, which
is the topic of the following section.

\begin{rem}
The idea to reduce the proof of Theorem \ref{int.t.main} to the uniform boundedness
of the operator family $(T_B(\lambda))_{\re\lambda> 0}$ is taken from 
the paper \cite{CioKey01} by \name{Cioranescu} and \name{Keyantuo}. These authors
employ the general theory of boundary values of holomorphic semigroups
as it is presented in \cite[Section 3.9]{ABHN}. That theory can be
incorporated into the general theory of functional calculus, but doing so here
would certainly take us too far astray. We decided to give an ad hoc
proof based on functional calculus methods in order to keep the
presentation as self-contained as possible and to demonstrate
once more the power of functional calculus theory.
\end{rem}

\section{The Transference Result}\label{s.tp}

Let us begin with some abstract considerations. Suppose one is given an
operator $T$ on a Banach space $X$ and wants to estimate its norm.
{\em Transference} means that one factorises the ``bad'' operator $T$ over a second Banach space
$Y$ via mappings $\iota: X \pfeil Y$, $S: Y \pfeil Y$ and $P: Y \pfeil X$, i.e.,
$ T = P S \iota$.
This means that the diagram
\[ \xymatrix{
Y  \ar@{->}[r]^{S}& Y\\
X\ar@{->}[u]^{\iota}\ar@{->}[r]^{T} & X\ar@{<-}[u]_{P}
}
\]
commutes. The operator $S$ is hopefully ``better'' than $T$ in the sense that
one has reasonable estimates on its norm. The factorisation 
leads to estimates of the form $\norm{T} \le c \norm{S}$.
It is possible in certain cases to keep $S$ fixed while varying $\iota, P$, thereby
improving on $c$.

A classical example is the transference principle by \name{Coifman} and \name{Weiss}
\cite{CoiWei,CoiWei77} in its
abstract form given by \name{Berkson}, \name{Gillespie} and \name{Muhly} \cite{BerGilMuh89b}.
It has the form
\[ \norm{ \int_\R U(s)x\, \mu(\diff{s})}\le M^2 
\norm{L_\mu}_{\Lin(\Ell{p}(\R;X))}
\quad \quad (\mu \in \eM(\R))
\]
where $U$ is a bounded $C_0$-group on a Banach space $X$,
$M := \sup_{s\in \R} \norm{U(s)}$ is its bound, and $L_\mu$ denotes
the convolution operator
\[  L_\mu := ( f \tpfeil \mu \ast f)
\]
on each space where it is meaningful. Such an estimate is particularly useful
if the Banach space is a UMD space, because then one can use Fourier multiplier
theory to estimate the norm of $L_\mu$. 
We aim at the analogous result when the group $U$ is replaced by a cosine
function.

\begin{thm}\label{tp.t.tpbc}
Let $(\Cos(t))_{t\in \R}$ be a bounded cosine function on a Banach space $X$, and let
$T_\mu$ be defined by
\[ T_\mu x = \int_{\R} \Cos(s)x\, \mu(\diff{s}) \quad \quad (x\in X, \mu \in \eM_e(\R)).
\]
Then 
\[ \norm{T_\mu} \le  5M^2 \norm{L_\mu}_{\Lin(\Ell{p}(\R;X))}
\quad \quad (\mu \in \eM_e(\R)),
\]
where $M := \sup_{s\in \R} \norm{\Cos(s)}$ and $p \in [1, \infty)$.
\end{thm}

\begin{proof}
Fix $p \in [1,\infty)$ and suppose first that $\mu$ has support within
the interval $[-N,N]$. We want to factorise $T_\mu$ over $Y := \Ell{p}(\R;X)$,
with $S = L_\mu$ being the convolution with $\mu$. 
Since $\mu$ is an
even measure, one has for $x\in X$ and $\abs{t} \le n$
\begin{align*}
\Cos(t) T_\mu x & =
\int_{-N}^N \Cos(t)\Cos(s)x\, \mu(\diff{s})
\\ & = \frac{1}{2} \left[
\int_{-N}^N \Cos(t-s)x\,\mu(\diff{s}) + \int_{-N}^N \Cos(t +s)x\, 
\mu(\diff{s}) \right]
\\ & =
\int_{-N}^N \Cos(t-s)x\, \mu(\diff{s})
= [\mu \ast (\iota_{n} x)](t) = (L_\mu \iota_n x)(t),
\end{align*}
where $\iota_n: X \pfeil Y$ is given by
\[ \iota_n(x) = [s \tpfeil \car_{[-(N+n), (N+n)]}(s)\Cos(s)x] \in \Ell{p}(\R;X) \quad \quad (x\in X).
\]
To determine $P_n: Y \pfeil X$ such that $T_\mu = P_n L_\mu \iota_n$, note that
by the defining identity for cosine functions (\ref{int.e.cos}) one
has $x = 2 \Cos(t)^2 x - \Cos(2t)x$ and hence
\[
x = \frac{1}{n} \int_{-n/2}^{n/2} x\,\diff{t} =
\frac{2}{n}\int_{-n/2}^{n/2} \Cos(t)^2x\,\diff{t} - \frac{1}{2n}\int_{-n}^n \Cos(t)x\,\diff{t}
\quad \quad (x\in X).
\]
If we replace $x$ by $T_\mu x$ in this identity and recall that
$\Cos(t)T_\mu x = (L_\mu \iota_n x)(t)$ for $\abs{t}\le n$, we see that
\[ T_\mu = P_n \, L_\mu\, \iota_n,
\]
where $P_n: Y \pfeil X$ is defined by
\[ P_n f := 
\frac{2}{n}\int_{-n/2}^{n/2} \Cos(t)f(t)\,\diff{t} - \frac{1}{2n}\int_{-n}^n 
f(t)\,\diff{t} \quad \quad (f\in Y).
\]
Let us estimate norms. One clearly has
\[
\norm{\iota_n x}_{\Ell{p}} = \left( \int_{-(n+N)}^{n+N} \norm{\Cos(s)x}^p
\,\diff{s}\right)^\frac{1}{p}
\le (2n+2N)^\frac{1}{p} M \norm{x} \quad \quad (x\in X),
\]
and hence $\norm{\iota_n}\le (2n + 2N)^\frac{1}{p} M$.
On the other hand, for $f\in Y = \Ell{p}(\R;X)$
\begin{align*}
\norm{P_n f}& \le \frac{2M}{n} \int_{-n/2}^{n/2}\norm{f(t)}\,\diff{t}
+ \frac{1}{2n}\int_{-n}^n \norm{f(t)}\,\diff{t}
\\& \le \frac{4M + 1}{2n} \int_{-n}^n \norm{f(t)}\,\diff{t}
\le \frac{4M + 1}{2n} (2n)^{1/p'} \norm{f}_{\Ell{p}([-n,n];X)}
\\ &
\le 5M   (2n)^{-1/p} \norm{f}_Y
\end{align*}
by H\"older's inequality. Hence $\norm{P_n} \le 5M (2n)^{-1/p}$. Combining these estimates
yields
\[ \norm{T_\mu} \le 5M^2 \left(1 + \frac{N}{n} \right)^\frac{1}{p} \norm{L_\mu}_{\Lin(Y)}.
\]
But $n$ was arbitrary, and so we can let $n \to \infty$ to obtain
\[ \norm{T_\mu} \le  5 M^2 \norm{L_\mu}_{\Lin(\Ell{p}(\R;X))}.
\]
As a last step we remove the support restriction on $\mu$.
For general even measure
$\mu$ the sequence of measures $\mu_n(\diff{s}) := \car_{[-n,n]}(s) 
\mu(\diff{s})$
converges to $\mu$ in the total variation norm. This implies convergence
$L_{\mu_n} \to L_\mu$ in $\Lin(\Ell{p}(\R;X))$ by Young's inequality, and 
$T_{\mu_n} \to T_\mu$ in $\Lin(X)$. The theorem is completely proved.
\end{proof}

To make effective use of Theorem \ref{tp.t.tpbc}, one has to have
good estimates for the norm of the convolution operators $L_\mu$. 
This is the topic of the next section.

\section{UMD Spaces and Proof of Main Theorem}\label{s.umd}

Let us recall the notion of a bounded Fourier multiplier. Fix $p \in [1, \infty)$.
A function $m \in \Ell{\infty}(\R)$ is called a {\em bounded} $\Ell{p}(\R;X)$-{\em Fourier multiplier},
if there is a constant $c=c(m, p, X)$ such that
\[  \big\|\Fourier^{-1}(m \fourier{f}\,)\big\|_{\Ell{p}(\R;X)} 
\le c \norm{f}_{\Ell{p}(\R;X)}
\]
for all functions $f$ belonging to the Schwartz class $\sfunk(\R;X)$. In this case
the operator $\ccalT_m$ given by
\[ \ccalT_m f := \Fourier^{-1}( m \fourier{f}\,) \quad\quad (f\in \sfunk(\R;X))
\]
extends to a bounded operator on $\Ell{p}(\R;X)$, and the function $m$ is
called the {\em symbol} of $\ccalT_m$.
We set
\[  \Mlt_p(\R;X) := \{ m \in \Ell{\infty}(\R) \suchthat \text{$m$ is
a bounded $\Ell{p}(\R;X)$-Fourier multiplier}\}
\]
with the norm $\norm{m}_{\Mlt_p(\R;X)} = \norm{\ccalT_m}_{\Lin(\Ell{p}(\R;X))}$.
The following lemma collects some useful facts.

\begin{lemma}\label{umd.l.fmp}
Let $p \in [1, \infty)$ and let $X$ be a Banach space. Then the following assertions hold.
\begin{aufzi}
\item If $\mu \in \eM(\R)$ then $\fourier{\mu} \in \Mlt_p(\R;X)$ and 
$\ccalT_{\fourier{\mu}} = L_\mu$ is convolution with $\mu$.
\item The space $\Mlt_p(\R;X)$ is a Banach algebra.
\item If $m \in \Mlt_p(\R;X)$ then for  
$\alpha, \beta, \gamma \in 
\R,\, \beta \not= 0$,
\[         m_{\alpha, \beta, \gamma}(t) :=
\expo{-i\alpha t} m(\beta t + \gamma) \,\,\in \,\,\Mlt_p(\R;X) 
\]
with $\norm{m_{\alpha, \beta, \gamma}}_{\Mlt_p(\R;X)} =
\norm{m}_{\Mlt_p(\R;X)}$.
\item $\Mlt_1(\R;X) = \Fourier \eM(\R)$.
\item If $X= H$ is a Hilbert space then $\Mlt_2(\R;H) = \Ell{\infty}(\R)$ with
$\norm{m}_{\Mlt_2(\R;H)} = \norm{m}_\infty$.
\end{aufzi}
\end{lemma}

\begin{proof}
These facts are standard and can be found in many books. 
\end{proof}

A Banach space $X$ is called  {\em HT-space} if the 
function
\[ h(t) := -i \sgn t \quad \quad (t\in \R)
\]
is a bounded $\Ell{2}(\R;X)$-multiplier.
The associated operator $\HT := \ccalT_{h}$ is 
called the {\em Hilbert transform}.
It is well known that
one may replace $\Ell{2}$ by any $\Ell{p}$, $p \in (1, \infty)$ in this definition.
Moreover, if $X$ is a HT-space then 
\[ \HT f(t) = \lim_{\epsilon \to 0} \int_{\epsilon \le \abs{s}\le 1/\epsilon} \frac{f(t-s)}{s}\,\diff{s}
\]
in the $\Ell{p}(\R;X)$-sense. (Actually, one can assert also convergence 
pointwise almost everywhere, but this is of no importance in this paper.)
After work of  \name{Burkholder} \cite{Bur81a} and \name{Bourgain} \cite{Bou83a},
the HT-property can be equivalently characterised by the so-called
UMD-property, involving {\em u}nconditional
{\em m}artingale {\em d}ifferences. We shall not make use of this characterisation,
but nevertheless use the name ``UMD space'', as this is now common.

\medskip

Suppose that $-A$ generates a bounded cosine function $(\Cos(s))_{s\in \R}$ on 
a UMD space $X$. The semigroup generated by $-B$ (which actually is
the same as $-A^{1/2}$, see Section \ref{s.fc2}) 
is given by
the Phillips calculus:
\[ T_B(\lambda) = T_{\mu_\lambda} \quad \text{with} \quad \fourier{\mu_\lambda}(t) =
\expo{-\lambda \abs{t}}\quad \quad (t\in \R). 
\]
By the transference principle (Theorem \ref{tp.t.tpbc}) 
one has 
\[ \norm{T_B(\lambda)} \le 5M^2 \norm{L_{\mu_\lambda}}_{\Lin(\Ell{2}(\R;X))} = 5 M^2 
\big\|\expo{-\lambda\abs{t}}\big\|_{\Mlt_2(\R;X)} \quad \quad (\re\lambda > 0).
\]
Now, writing $\lambda = \alpha + is$ with $\alpha > 0, s\in \R$:
\[ \expo{-\lambda\abs{t}} = \expo{-\alpha \abs{t}} (\cos(st) + h(t)\sin(st)) \quad \quad (t\in \R) 
\]
as a simple computation shows. The symbols $\expo{-\abs{t}}, \sin(t), \cos(t)$ are
all in $\Fourier\eM(\R)$, hence in $\Mlt_2(\R;X)$, and also  
$h\in \Mlt_2(\R;X)$, since $X$ is a UMD space. 
Thus it follows from Lemma \ref{umd.l.fmp} that
the family $(\expo{-\lambda\abs{t}})_{\re \lambda > 0}$ is uniformly bounded
in $\Mlt_2(\R;X)$. Hence we have proved the following statement.

\begin{prop}
Let $-A$ generate a bounded cosine function $(\Cos(s))_{s\in \R}$ on the UMD space $X$.
Then the holomorphic
semigroup $(T_B(\lambda))_{\re\lambda >0}$ is uniformly bounded.
\end{prop}

As was explained in Section \ref{s.fc1}, this completes the proof of the first
part of  Theorem \ref{int.t.alt}. 
In the next section we shall provide a second
approach to this result.


\section{A Second Approach}\label{s.rep}

We still postpone the proof that $B= A^{1/2}$ and now present a second
approach to the group $U$, based on a different approximation 
of the exponential function. 
Consider the standard formula
\[ \expo{-is\abs{t}} = \cos(s t) - i \sgn(t)\sin(st) \quad \quad (s,t\in \R).
\]
Inserting $B$  yields
\begin{equation}\label{rep.e.euler}
 U(s)= \expo{-isB} = \Cos(s) - i B\Sin(s) \quad \quad (s\in \R)
\end{equation}
where, for $s\in \R$, 
\[ \Sin(s) = \left(\frac{\sin(s\abs{t})}{\abs{t}}\right)(B)=
\left(\int_0^s \cos(r\abs{t})\, \diff{r}\right)(B)
= \int_0^s \Cos(r)\,\diff{r}
\]
is the associated sine function \cite[page 210]{ABHN}. (One has indeed equality in 
(\ref{rep.e.euler}), 
by general functional calculus \cite[Theorem 1.3.2]{HaaFC} 
and  since $\Cos(s)$ is a bounded operator.)
The alternative approach to Theorem \ref{int.t.main} now proceeds
via the following representation result. 

\begin{thm}\label{rep.t.rep}
Let $-A$ generate a bounded cosine function $\Cos$ on a Banach space $X$,
and let $S(s) := [\sgn(t)\sin(st)](B) = B\Sin(s)$, $s\in \R$. Then 
\begin{equation}\label{rep.e.rep}
S(s) x = 
\frac{1}{\pi}\,\, \PV-\int_\R \frac{\Cos(s-r)x}{r} \,\diff{r} \quad \quad 
(s\in \R).
\end{equation}
for all $x\in \DD(A)$.
In the case that $X$ is a UMD space, $\sup_{s\in \R}\norm{S(s)} < \infty$, 
and the representation {\upshape (\ref{rep.e.rep})}  
holds for all $x\in X$.
\end{thm}

The formula (\ref{rep.e.rep}) is due to \name{Fattorini} \cite{Fat69b}.
The extension in the UMD case was established by 
\cite[Proof of Theorem~2.5]{CioKey01};
in their proof the authors make use of Burkholder's result that in a UMD
space the Hilbert transform converges almost everywhere. Moreover,
some intricate measure-theoretic arguments are also needed. 
Our approach does not need more than the mere definition of UMD space,
as well as some mild Fourier analysis and 
the functional calculus constructed in Section \ref{s.fc1}.

\medskip

In the proof of Theorem \ref{rep.t.rep} we shall have occasion to use the function
\[ H(t) := \int_0^1 \frac{\sin(st)}{s} \, \diff{s} \quad \quad (t\in \R),
\]
which is odd, bounded and has bounded derivative. A classical fact is the
identity
\[  \lim_{t \to \pm \infty} H(t) = \lim_{t\to \pm \infty} \int_0^t 
\frac{\sin s}{s}\, \diff{s} =
\frac{\pm \pi}{2}.
\]
Let us also introduce the function 
\[ F(t) := \sgn(t)H(t) - \frac{\pi}{2}  = \int_{\abs{t}}^\infty \frac{\sin r}{r}\, \diff{r}\quad \quad (t\in \R)
\]
which is continuous, even and vanishes at infinity, and finally
\[ G_{s,c}(t) := \sgn(t)\sin(st)F(ct) \quad \quad (t\in \R, s\in \R, c>0).
\]

\medskip
Now look at the stated convergence of the principal value integral 
(\ref{rep.e.rep}).
A simple calculation yields
\begin{align}\label{rep.e.gen}
 \int_{a\le\abs{r}\le b} &  \frac{\cos[(s-r)t]}{r}\, \diff{r}
 \,\,=\,\, 2 \sin(st) \int_a^b \frac{\sin(rt)}{r} \, \diff{r} 
\nonumber \\ & \,\,=\,\,2 \sin(st)[ H(bt) - H(at)]
\,\,=\,\,2 \sgn(t)\sin(st)[F(bt) - F(at)] 
\nonumber \\ & \,\,=\,\,2 G_{s,b}(t) - 2 G_{s,a}(t)
\end{align}
for $t\in \R$. We have to look what happens as $b$ tends to $\infty$ and $a$ tends to $0$.

\begin{lemma}\label{rep.l.FGH}
The following assertions hold.
\begin{aufzi}
\item $F \in \Fourier(\Ell{1}(\R))$.
\item For all $s\in \R$ and $c > 0$, the function $G_{s,c}$ 
is in $\calE(\R_+)$.
\item $G_{s,c}\to 0$ in $\calE(\R_+)$ as $c\to \infty$,  
in the sense that there exist 
measures $\mu_{s,c} \in \eM(\R)$ such that
$\CT\mu_{s,c} = G_{s,c}$ and   
$\mu_{s,c} \to 0$ in $\eM(\R)$.
\item For fixed $s\in \R$
\[ \frac{G_{s,c}(t)}{1 + t^2} \to \frac{-(\pi/2)\sgn(t)\sin(st)}{1+t^2}\quad \quad \text{as $ c \searrow 0$}
\]
in $H^1(\R)$ (as functions of $t$).
\end{aufzi}
\end{lemma}

\begin{proof}
a) Since 
\[ F(t) = - \int_{\abs{t}}^\infty \frac{\sin s}{s}\, \diff{s} \quad \quad (t\in \R)
\]
we have $\abs{t} F'(t) = \sin(t)$, whence $F'\in \Ell{2}(\R)$. For $t \not=0$, integration by parts yields
\[ F(t) = - \int_{\abs{t}}^\infty \frac{\sin r}{r}\, \diff{r} = \frac{\cos t}{\abs{t}}
+ \frac{1}{\abs{t}}\int_1^\infty \frac{\cos(tr)}{r^2}\, \diff{r}.
\]
This shows that also $\abs{t}F(t)$ is bounded and hence $F\in \Ell{2}(\R)$. We conclude that
$F\in H^1(\R)$ and therefore in $\Fourier (\Ell{1}(\R))$, by Bernstein's Lemma \ref{fc1.l.bern}.

b), c)\ Using the above integration by parts, we find
\[ G_{s,c}(t) = \left( \frac{\sin(st)}{t}\right) \left(\frac{\cos(ct)}{c} + \int_c^\infty \frac{\cos(tr)}{r^2}
\, \diff{r}\right) \quad \quad (t\in \R).
\]
Note that  $\CT(\car_{[0,s]}) = \sin(st)/t$ and the second factor is the cosine transform of
$\mu_c := c^{-1}\delta_c + \car_{[c, \infty)}r^{-2}\, \diff{r}$. 
This proves b), and c)  follows readily  
since  obviously $\mu_c \to 0$ in $\eM(\R)$ as $c\to \infty$.

d) It is easily seen that $\sgn(t)\sin(st)(1+t^2)^{-1} \in H^1(\R)$. Furthermore,
\[ G_{s,c}(t) + (\pi/2)\sgn(t)\sin(st) = \sin(st)H(ct)
\]
and it is likewise easy to see that
\[ \frac{\sin(st)H(ct)}{1+t^2} \to 0 \quad \quad (c\searrow 0)
\]
in $H^1(\R)$, as functions of $t\in \R$. This completes the proof.
\end{proof}

\medskip

Returning to our starting point, we may insert the operator $B$ by means of the functional calculus
and obtain (for fixed $s\in \R$ and $0 < a < b < \infty$)
\[  \frac{1}{\pi} \int_{a \le \abs{r} \le b}  \frac{\Cos(s-r)}{r}\, \diff{r} =
\frac{2}{\pi}( G_{s,b}(B) - G_{s,a}(B)).
\]
By c) of Lemma \ref{rep.l.FGH} one has $\lim_{b \to \infty} G_{s,b}(B) = 0$ in norm,
and by d) of Lemma \ref{rep.l.FGH} together with the convergence lemma 
(Lemma \ref{fc1.l.cl}) one has
\[ - \frac{2}{\pi} G_{s,a}(B)x \to (\sgn(t)\sin(st))(B)x = S(s)x
\]
for all $x\in \DD(A)$. Moreover, the convergence is true for all $x\in X$
in the case that $\sup_{0< a\le 1} \norm{G_{s,a}(B)} < \infty$.

\medskip
To complete the proof of Theorem \ref{int.t.main} suppose that 
$X$ is a UMD space.
It suffices to establish the uniform boundedness
\begin{equation}
\sup_{0< a < b} 
\norm{ \int_{a \le \abs{r} \le b}  \frac{\Cos(s-r)}{r}\, \diff{r} }
< \infty
\end{equation}
or, equivalently,  
\[ \sup \{ \norm{G_{s,c}(B)} \suchthat s\in \R, c > 0\}.
\] 
By the transference principle (Theorem \ref{tp.t.tpbc}) and (\ref{rep.e.gen}) 
it is sufficient
to show that the family of $\Ell{2}(\R;X)$-Fourier multiplier 
operators with symbols
\[ \sgn(\cdot)\sin(s\,\cdot )F(c\,\cdot )\quad \quad (s \in \R, c > 0)
\]
is uniformly bounded. Now $F \in \Fourier(\Ell{1}(\R))$ by a) of Lemma 
\ref{rep.l.FGH}, and 
$\sgn(\cdot)$ is a bounded Fourier multiplier since $X$ is UMD.
Hence the claim follows from c) of Lemma \ref{umd.l.fmp}.

\section{Proof that $B = A^{1/2}$ and Compatibility of Functional Calculi}\label{s.fc2}

Let $-A$ be the generator of a uniformly bounded cosine function
$\Cos$ on the Banach space $X$, and let $B = \Phi(\abs{t})$, where
$\Phi$  is the functional calculus constructed in Section \ref{s.fc1}.
As a first step, let us make sure that the operator $A^{1/2}$ is
defined by classical theory. This amounts to proving that $A$ is
a {\em sectorial operator}. 

\smallskip

We shall 
use the following abbreviation.
For $\omega \in [0, \pi]$  
\[ \sector{\omega} := 
\begin{cases}
\{ z \not= 0 \suchthat \abs{\arg z} < \omega\}, & \omega \in (0, \pi]\\
    (0, \infty), & \omega = 0
\end{cases}
\]
denotes the horizontal sector of angle $2\omega$, symmetric about the
positive real axis. For background and terminology of sectorial and
strip-type operators we refer to \cite[Chapter 3]{HaaFC}.

\begin{lemma}\label{fc2.l.stsec}
Let $-A$ be the generator of a cosine function, uniformly bounded by $M\ge 1$, 
on a Banach space $X$. Then
$\sigma(A) \subset \R_+$ and $A$ satisfies the resolvent estimate
\begin{equation}\label{fc1.e.fund}
 \norm{R(\lambda^2,A)} \le \frac{M}{\abs{\lambda} \abs{\im \lambda}}
\quad\quad (\lambda \notin \R).
\end{equation}
Furthermore, if $A$ is any operator satisfying {\upshape (\ref{fc1.e.fund})},
then $A$ is sectorial of angle $0$ and $A^{1/2}$ is both
sectorial of angle $0$ and
(strong) strip-type of height $0$.
More precisely, there is $\tilde{M}\ge 0$ such that
\[  \norm{R(\lambda,A^{1/2})} \le \frac{\tilde{M}}{\abs{\im \lambda}} \quad
\quad(\lambda \notin \R).
\]
\end{lemma}

\begin{proof}
The estimate (\ref{fc1.e.fund}) follows directly from the representation
(\ref{int.e.gen}). (Note that one has replace $A$ by $-A$ there and that
the formula extends by holomorphy to all $\re\lambda > 0$.) 

By  (\ref{fc1.e.fund}) we have
\[ \norm{\lambda^2 R(\lambda^2,A)} \le M \frac{\abs{\lambda}}{\abs{\im \lambda}}
\quad \quad (\lambda \notin \R).
\]
Now if $\phi \in (0, \pi)$ and $\mu \in \C \ohne \cls{\sector{\phi}}$,
then there is $\lambda \notin \R$ such that $\lambda^2 = \mu$ and
$\phi/2 < \abs{\arg \lambda} < \pi - \phi/2$. Hence
\[ \norm{\mu R(\mu,A)} \le \frac{M}{\abs{\sin (\arg \lambda)}} \le 
\frac{M}{\sin(\phi/2)} \quad \quad
(\mu \notin \cls{\sector{\phi}}).
\]
This proves that $A$ is sectorial of angle $0$. It follows by general theory
of (fractional powers of) sectorial operators \cite[Chapter 3]{HaaFC} that
$A^{1/2}$ is well-defined and again sectorial of angle $0$.  
Choose $M'$ such that
\begin{equation}\label{fc2.e.sst1}
 \norm{\lambda R(\lambda,A^{1/2})} \le M' \quad \quad (\re \lambda < 0).
\end{equation}
Now, if $\lambda \notin \R$ write
\[ \frac{1}{\lambda - z} - \frac{1}{(-\lambda) - z}
= \frac{2 \lambda}{\lambda^2 - z^2}.
\]
Inserting $A^{1/2}$ yields
\begin{equation}\label{fc2.e.sst2} 
R(\lambda, A^{1/2}) = 2 \lambda R(\lambda^2,A) + R(-\lambda,A^{1/2}).
\end{equation}
If $\re \lambda \le 0$ we have by (\ref{fc2.e.sst1})
\[ \norm{R(\lambda,A^{1/2})} \le \frac{M'}{\abs{\lambda}}
\le 
\frac{M'}{\abs{\im \lambda}}.
\]
If $\re \lambda \ge 0$ then $\re (-\lambda) \le 0$ and hence
\[ \norm{R(\lambda,A^{1/2})} \le 2\norm{\lambda R(\lambda^2,A)}
+ \norm{R(-\lambda,A^{1/2})}
\le \frac{2M}{\abs{\im \lambda}} + \frac{M'}{\abs{\im \lambda}}
\]
by (\ref{fc2.e.sst2}).
So the assertion holds with $\tilde{M} := M' + 2M$.
\end{proof}

The aim of this section is
not only to show that $A^{1/2} = B$ (as defined in Section \ref{s.fc1}) but
also to prove that the functional calculus $\Phi$ from Section \ref{s.fc1}
is a proper extension of the functional calculi for $A^{1/2}$ as
a sectorial operator and as a strip-type operator.

\medskip
Let us recall some definitions and results from \cite[Chapter 2]{HaaFC}. 
For $\phi \in (0, \pi]$ we define
\[ H^\infty_0(\sector{\phi}) := 
\big\{ f \in H^\infty(S_\phi) \suchthat 
\exists M, s> 0 : \abs{f(z)} \le M \min(\abs{z}^s, \abs{z}^{-s})\big\}.
\]
The {\em primary} functional calculus for $A^{1/2}$ as a sectorial operator is given by 
\[ \Psi(f) := \frac{1}{2\pi i} \int_\Gamma f(z)R(z,A^{1/2})\, \diff{z}
\]
for $f\in H^\infty_0(\sector{\phi})$, where $\phi \in (0, \pi)$ 
and $\Gamma = \rand\sector{\phi'}$ with $0 < \phi' < \phi$ being arbitrary.
Since these functions are holomorphic, they are determined completely
by their restrictions to $\R_+$ and we will tacitly perform this
restriction whenever it is convenient.

\begin{prop}
Let $-A$ be the generator of a uniformly bounded cosine function on 
a Banach space $X$. Let $\Phi$ denote the functional calculus defined in 
Section \ref{s.fc1}, with $B:= \Phi(\abs{t})$. Let $\phi \in (0, \pi)$
and $f\in H^\infty_0(\sector{\phi})$. Then $f\in \calE(\R_+)$ and 
$\Phi(f) = \Psi(f)$.   
\end{prop}

\begin{proof}
Fix $\phi' \in (0, \phi)$ and let $\Gamma := \rand\sector{\phi'}$,
$\Gamma_+ = \R_+ e^{i\phi}$, $\Gamma_- := \R_+ e^{-i\phi}$.
Then by Cauchy's theorem
\[ \frac{1}{2\pi i} \int_\Gamma f(z) R(-z, A^{1/2})\, \diff{z} = 0.
\]
Hence
\begin{align*}
f(A^{1/2}) & = \frac{1}{2\pi i} \int_\Gamma f(z) R(z,A^{1/2})\, \diff{z}
\\ & =
\frac{1}{2\pi i} \int_\Gamma f(z) [R(z,A^{1/2}) - R(-z,A^{1/2})]\, \diff{z}
\\ & =
\frac{1}{\pi i} \int_\Gamma f(z)z R(z^2,A)\, \diff{z}
\\ & = 
\frac{1}{\pi i} \int_{\Gamma_+} f(z)z R((-iz)^2,-A)\, \diff{z}
-
\frac{1}{\pi i} \int_{\Gamma_-} f(z)z R((iz)^2,-A)\, \diff{z}
\\ & =
\frac{1}{\pi} \int_{\Gamma_+} f(z) (-iz) R((-iz)^2,-A)\, \diff{z}
+
\frac{1}{\pi} \int_{\Gamma_-} f(z) (iz) R((iz)^2,-A)\, \diff{z}
\\ & =
\frac{1}{\pi} \int_{\Gamma_+} f(z) \int_0^\infty e^{izs}\Cos(s)\,\diff{s}\, \diff{z}
+
\frac{1}{\pi} \int_{\Gamma_-} f(z) \int_0^\infty e^{-izs}\Cos(s)\,\diff{s}\, \diff{z}
\\ & =
\int_\R \left[
\frac{1}{2 \pi} \int_{\Gamma_+} f(z) e^{z\abs{s}}\, \diff{z} +
\frac{1}{2 \pi} \int_{\Gamma_-} f(z) e^{-z\abs{s}}\, \diff{z} 
\right]\, \Cos(s)\, \diff{s}.
\end{align*}
It is routine to check that the function
\[ g(s) := \frac{1}{2 \pi} \int_{\Gamma_+} f(z) e^{z\abs{s}}\, \diff{z} +
\frac{1}{2 \pi} \int_{\Gamma_-} f(z) e^{-z\abs{s}}\, \diff{z}
\quad \quad (s\not= 0)
\]
is  in $\Ell{1}(\R)$. By specializing $X = \C$ and $A = t^2 \ge 0$ obtain
\[ \CT(g(s)\diff{s})(t) = f(t) \quad \quad (t \in \R).
\]
This finally yields $f(A^{1/2}) = \Phi(f)$ as claimed. 
\end{proof}

Consider now the function $e(z) := (1+z)^{-1}$ and note that
\[ e(z) - \frac{1}{1 + z^2} = \frac{z^2 - z}{(1+z)(1+z^2)}
= \frac{z(z-1)}{(1+z)(1+z^2)} =: f(z) \in H^\infty_0(\sector{\phi})
\]
for any $ \phi < \pi/2$. Therefore,
\[ e(A^{1/2}) = f(A^{1/2} + (1+A)^{-1}
= f(B) + (1+A)^{-1} = e(B).
\]
This implies that $A^{1/2} = B$. Moreover, a look on the  construction of the
sectorial functional calculus for $A^{1/2}$ in \cite[Chapter 2]{HaaFC} 
makes it clear that the functional calculus $\Phi$ is a 
proper extension of it. 
(By \cite[Proposition 1.2.7]{HaaFC} one can then conclude, that the compatibility for the primary
calculi carries over to the unbounded extensions.)

\bigskip

\noindent
{\em Acknowledgements}\\
This article was completed while the author enjoyed a two-month stay
at the Institute of Mathematics of the Polish Academy of Sciences in Warsaw
and at the Nikolaus-Copernicus University in Toru\`n, Poland. The stay was
supported by the EU Marie Curie ``Transfer of Knowledge'' project
``Operator Theory Methods for Differential Equations'' (TODEQ), and
the the author is grateful for this support and the kind invitation
by J.~Zem\'anek (Warsaw) and Y.~Tomilov (Toru\`n).

The author is also grateful to Bernhard Haak
(Bordeaux) who made valuable comments on a previous version of this article.

\bibliographystyle{plain}
\def\cprime{$'$} \def\cprime{$'$} \def\cprime{$'$} \def\cprime{$'$}

\end{document}